\documentclass[11pt]{article}
\usepackage{amssymb,amsbsy,amsmath,amsfonts,amssymb,amscd,colordvi}
\usepackage{amsthm}
\usepackage{epsfig}
\usepackage{graphicx}
\usepackage[english]{babel}
\usepackage{times}
\usepackage{euscript}
\usepackage{color}
\def \ve{\varepsilon}
\def \eps{\epsilon}

\def \be{\begin{equation}}
\def \ee{\end{equation}}

\def \eqdef{\mathop{=}\limits^{\hbox{\tiny def}}}
\usepackage{color}

\newtheorem{theorem}{Theorem}[section]
\newtheorem{lemma}{Lemma}[section]
\newtheorem{definition}[theorem]{Definition}
\newtheorem{remark}{\rm Remark\/}

\def \trait (#1) (#2) (#3){\vrule width #1pt height #2pt depth #3pt}
\def \fin{\hfill
	\trait (0.1) (5) (0)
	\trait (5) (0.1) (0)
	\kern-5pt
	\trait (5) (5) (-4.9)
	\trait (0.1) (5) (0)
\medskip}
\font \twbbb= msbm10 scaled \magstep0                 
\font \tenbbb= msbm7 scaled \magstep0                 
\newfam\bbbfam             
\textfont\bbbfam=\twbbb \scriptfont\bbbfam=\tenbbb    %

\begin{document}

\renewcommand{\theequation}{\thesection .\arabic{equation}}

\renewcommand{\baselinestretch}{1.2}
\title{Homogenization of immiscible compressible two--phase  flow in random porous media}
\author{B. Amaziane$^{1}$, L. Pankratov$^{1,2}$, A. Piatnitski$^{3,4}$}
\date{}
\maketitle

$^1$ Laboratoire de Math\'ematiques et de leurs Applications,  CNRS-UMR 5142,
Universit\'e de Pau, Av. de l'Universit\'e, 64000 Pau, France.
E--mail: {\tt brahim.amaziane@univ-pau.fr}
\medskip

$^2$ Laboratory of Fluid Dynamics and Seismics,
Moscow Institute of Physics and Technology, 9 Institutskiy per., Dolgoprudny, Moscow Region, 141700, Russian Federation.\\
E--mail: {\tt leonid.pankratov@univ-pau.fr}
\medskip

$^3$The Arctic University of Norway, campus Narvik, Postbox 385, Narvik, 8505, Norway

\smallskip
$^4$
 Institute for Information Transmission Problems of RAS, Bolshoy Karetny per., 19, Moscow, 127051, Russia.\ \
E--mail: {\tt apiatni@iitp.ru}

\begin{abstract}
The paper deals with homogenization of a model problem describing an immiscible compressible two-phase flow in random statistically homogeneous porous media.  We derive the effective (macroscopic) problem and prove the convergence of
solutions.

Our approach relies on stochastic two-scale convergence techniques, the realization-wise notion of stochastic two-scale
convergence being used. Also, we exploit various a priori estimates as well as monotonicity and compactness arguments.
\end{abstract}

{\bf AMS Subject Classification:} 35B27; 35K65; 74Q15; 76M50; 76S05; 76T10;

{\bf Keywords:} heterogeneous porous media; homogenization; random porous media; two--phase flow.

\section{Introduction}
\label{intro}
The problem of description of two-phase and multiphase flows through highly heterogeneous media is faced in many branches of engineering and applied sciences
such as groundwater hydrology, petroleum engineering, environmental sciences.
More recently, multiphase flow attracted an essential interest of engineers and
researches dealing with deep geological repository for nuclear waste and 
and for storage sequestration.

The crucial goal here is to isolate the radioactive waste from the biosphere. Repositories for the disposal of long-lived nuclear waste generally rely on a multi-barrier system that typically consists of the natural geological barriers
provided by  the host rock and its surroundings and an engineered barrier system
that comprises engineered materials placed within a repository.

In the framework of designing nuclear waste geological repositories problems of
two-phase flow of water and gas naturally appear, for further details  see \cite{Norr2009}, \cite{Shaw}.  
Recent studies have shown that in such engineered systems  considerable amounts of gases, mainly hydrogen, might be produced for instance due to the corrosion of metallic components used in the repository design. It is than important to evacuate the gas phase and avoid
overpressure, thus preventing mechanical damages.
Therefore, the appearance and transport
of a gas phase is  an important issue of concern with regard to the capability
of the engineering and natural barriers to evacuate the gas phase.  Thus, it is  necessary to take these issues into account when estimating the performance of a geological repository, see \cite{SengerEtAl}, \cite{CroiseEtAl},  \cite{Shaw} and references therein for further details.

As was said above the main source of gas is the corrosion phenomena of steel lines,
waste containers, ets. However, there are other sources of gas such as the water radiolysis by radiation emitted by the nuclear waste  and the microbial activity
that generates some methane and carbon dioxide.  Typically, the proportion of hydrogen in the total mass of produced gases is above 90\%.

Since in the subsurface permeability heterogeneity occurs in many different length scales, numerical models of flow cannot, in general, resolve all the plurality of scales. Therefore, approaches based on upscaling or homogenization are required to represent the effect of subgrid scale variations on larger scale flow.

The problem of homogenization of multiphase flow through heterogeneous porous media
has quite a long history, a number of methods and approaches have been elaborated. There is a vast literature on this topic. Here we will merely mention some references to mathematical homogenization results for flow and transport in porous media. There were a number of works devoted to the qualitative theory of systems of equations  describing incompressible immiscible two-phase flow in porous media. Among them are
\cite{ant-kaz-mon1990}, \cite{Arbogast-92}, \cite{Chen-I,Chen-II}, \cite{Kro-Luck},    where the questions of existence, uniqueness and regularity of a weak solution were investigated.
For the results on  homogenization of incompressible single phase flow through heterogeneous media we refer to \cite{Bourgeat-Gip2004}, \cite{Bourgeat-Marusic2005}, \cite{Bourgeat-Marusic2010}, \cite{Bourgeat-Piat2010}, \cite{Gloria-Etal} and the references therein.
In the works \cite{HOS-2012} and \cite{hor}   homogenization problems for a  compressible miscible flow in porous media have been studied.
Important qualitative results such as existence and regularity of weak solutions for compressible immiscible two-phase flow in heterogeneous
porous media were obtained in  \cite{CG-MS1}, \cite{CG-MS2}, \cite{CG-MS4}, and in the case of discontinuous capillary pressure
in the work \cite{app-1}.

The problem of homogenization of compressible immiscible two-phase flow in porous media with a periodic microstructure has been addressed in our previous works \cite{our-siam}, \cite{app-2}, \cite{app-nhm17}, in the latter work the case of several types of rocks has been considered.

The mentioned articles considered homogenization problems in  porous media with a periodic microstructure. In particular,  the rigorous homogenization results  obtained in these papers are valid under the assumption that the corresponding porous medium is periodic or locally periodic.
 Although the results obtained for systems with periodic coefficients
 provide an important information on the effective behaviour of
 the two-phase flow of interest,  the description of the effective behaviour of the flow based on the periodicity assumption usually cannot be accurate ,
 except for the case of artificial materials.
 In the case of natural materials the assumption that the porous media are random statistically homogeneous is much more realistic and allows to provide more accurate description of the effective characteristics.

This paper focuses on modelling and effective description of immiscible compressible two-phase flows through heterogeneous reservoirs
with random statistically homogeneous geometric structure, in the framework of the geological disposal of nuclear waste.  We will be concerned with a nonlinear system of convection-diffusion equations in a domain modeling the flow and transport of immiscible compressible fluids through heterogeneous porous media, taking into account capillary
and gravity effects.  More precisely, we will assume that there is an incompressible wet phase (water) and a compressible gas
phase (hydrogen) in the context of gas migration through engineered and geological barriers for a deep repository for nuclear waste,
for more details see \cite{app-1}.

Here we consider a single rock-type model. The original microscopic model is defined in a domain with random statistically
homogeneous ergodic microstructure. In our context it means that both the porosity and the absolute permeability are rapidly oscillating statistically homogeneous random
functions of the microscopic variable $y=x/\varepsilon$, where $x$ is the macroscopic variable, and  $\varepsilon>0$ is a small parameter
 that represents the characteristic length scale of the medium. The problem is formulated in terms of the wetting saturation phase (water phase) and nonwetting pressure phase (gas phase).   The corresponding system of equations is derived from the mass conservation laws of both fluids on the one side   and from the relations between the velocities and the pressure gradients as well as the gravitational forces.
These relations are provided by the Darcy-Muscat law. The resulting system consists of a nonlinear parabolic equation for the gas pressure
coupled with a degenerate parabolic convection-diffusion equation for the water saturation, both equations are subject to appropriate boundary
and initial conditions.

In this system the diffusion operators degenerates due to the capillary effects, the degeneracy of this type can be observed both in compressible and incompressible flows. Another type of degeneracy occurs in the region where the gas saturation vanishes.
In this region  the gas density cannot be determined by its evolution since the gas phase is not presented.

The degeneracy and strong coupling of the equations in the system of interest make the proof of homogenization result rather involved
especially in the case of random coefficients.  In particular, we are not able to obtain uniform estimates for the gradients of the phase pressures.
To overpass this difficulty we reformulate the studied problem in terms of the global pressure and the saturation. This leads to a less strong coupling between the equations of the system.  However, we still do not have uniform estimates for the saturation gradient. In addition, due to degeneracy, solutions do not possess much regularity.  As a result,
passing to the limit in the studied system is not direct and requires rather delicate arguments.

Homogenization problem for a two-phase incompressible flow in a random medium was treated successfully in \cite{BKM95}.
However, to our best knowledge, rigorous homogenization results for an immiscible compressible multi-phase flow in a random medium
are missed in the existing literature.

The paper is organized as follows. Section \ref{micromodel} deals with problem setup. We describe the physical model,
introduce the corresponding system of equations and provide the assumptions on the data.

In Section \ref{for-exist-result} we recall the results on the existence of a solution and obtain a number of a priori estimates.

Finally, in Section \ref{sec_hm} we formulate and prove the homogenization theorem.

\section{Problem setup}
\label{micromodel}


We consider an immiscible compressible two-phase flow process in a heterogeneous porous medium with a random statistically homogeneous  microstructure. We suppose that the porous medium occupies a porous reservoir  $\mathcal{Q} \subset \mathbb{R}^d$, $d = 1, 2, 3$,  being a bounded connected Lipschitz domain.
We then assume that there are  general wetting  and a non-wetting phases, in particular examples it might  be the phases water and gas.  For presentation simplicity we assume that there is no a source/sink term.

For further details we refer the reader to \cite{ant-kaz-mon1990, GC-JJ,Chen-Huan-Ma-06,HR}.


To introduce a random microstructure we assume that $(\Omega,\mathcal{F},\mathbf{P})$ is a standard probability space
equipped with an ergodic dynamical system $\mathcal{T}_x$, $x\in\mathbb R^d$,  that is\\[-2mm]
\begin{itemize}
  \item
  $
  \mathcal{T}_{x+y}=\mathcal{T}_x\circ \mathcal{T}_y,\quad x\,,y\in\mathbb R^d,\qquad \mathcal{T}_0={\rm Id};
  $
  \item
  $ \mathbf{P}(\mathcal{T}_x(A))=\mathbf{P}(A)\qquad \hbox{for all } x\in \mathbb R^d, \ \ A\in\mathcal{F}$;
  \item
  $  \mathcal{T}_\cdot \,:\, \mathbb R^d\times\Omega\mapsto\Omega $ is a measurable mapping from  $\mathbb R^d\times\Omega$
  to $\Omega$, $\mathbb R^d$ with being equipped with the Borel $\sigma$-algebra.\\[-2mm]
\end{itemize}
The ergodicity of $\mathcal{T}_\cdot$ means the probability of any set $A\in\mathcal{F}$ which is invariant with respect to $\mathcal{T}_x$,
$x\in\mathbb R^d$, is equal to zero or one.

Also we assume that there are a positive random variable $\boldsymbol{\Phi}=\boldsymbol{\Phi}(\omega)$ and a positive definite random matrix $\boldsymbol{K}=\boldsymbol{K}(\omega)$ and denote
$$
\Phi(x)=\boldsymbol{\Phi}(\mathcal{T}_x\omega),\qquad  K(x)=\boldsymbol{K}(\mathcal{T}_x\omega),
$$
that is $\Phi(x)$ and $K(x)$ are realizations of  $\boldsymbol{\Phi}$ and $\boldsymbol{K}$.

We then scale this
random structure with a parameter $\ve$ which represents the ratio of the typical size of local inhomogeneities  to the size of the whole region
$\mathcal{Q}$,  and definite  the porosity and the absolute permeability tensor by
\begin{equation}\label{reali}
\Phi^\ve(x)=\Phi\Big(\frac x\ve\Big)=\boldsymbol{\Phi}(\mathcal{T}_\frac x\ve\omega),\qquad  K^\ve(x)=K\Big(\frac x\ve\Big)=\boldsymbol{K}(\mathcal{T}_\frac x\ve\omega),
\end{equation}
we assume that $0<\ve \ll 1$ is a small parameter that tends to zero.


Prior to writing down the equations of the model, we introduce the following notation.
Let  $S^\ve_w = S^\ve_w(x, t)$ and $S^\ve_g = S^\ve_g(x, t)$
be  the saturations of the wetting and non-wetting phases, respectively;
$k_{r,w} = k_{r,w}(S^\ve_w)$ and $k_{r,g} = k_{r,g}(S^\ve_g)$ stand for the relative permeabilities of the corresponding phases;
$p^\ve_w = p^\ve_w(x,t)$, $p^\ve_g = p^\ve_g(x,t)$
are their pressures;
and {$\varrho_w$, $\varrho_g$} are the corresponding densities.  We then fix an arbitrary time interval $[0,T]$, $T>0$, and denote $\mathcal{Q}_T = \mathcal{Q} \times ]0, T[$.
The equations of the mass balance of each phase read (see, e.g., \cite{GC-JJ,Chen-Huan-Ma-06,HR})
\begin{equation}
\label{debut1}
\left\{
\begin{array}[c]{ll}
\displaystyle
\Phi^\ve(x) \frac{\partial}{\partial t}(S^\ve_w\, \varrho_w(p^\ve_{w})) +
{\rm div}\, (\varrho_w(p^\ve_{w}) \, \vec q^{\,\,\ve}_w) = 0
\quad {\rm in}\,\, \mathcal{Q}_T; \\[3mm]
\displaystyle
\Phi^\ve(x) \frac{\partial}{\partial t}(S^\ve_g\, \varrho_g(p^\ve_{g})) +
{\rm div}\, (\varrho_g(p^\ve_{g}) \, \vec q^{\,\,\ve}_g) = 0
\quad {\rm in}\,\,\mathcal{Q}_T;  \\[3mm]
\end{array}
\right.
\end{equation}
where  the velocities of the wetting and non-wetting phases
 $\vec q^\ve_{w}$ and $\vec q^\ve_{g}$ satisfy the Darcy-Muskat law:
\begin{equation}
\label{eq.qw}
\vec q^\ve_{w}= -K^\ve(x) \lambda_{w}(S^\ve_{w})
\bigg(\nabla p^\ve_{w} - \varrho_w(p^\ve_{w}) \vec{g}\bigg)
\qquad {\rm with}\,\, \lambda_{w}(s) = \frac{k_{r,w}}{\mu_{w}}
(s);
\end{equation}
\begin{equation}
\label{eq.qg}
\vec q^\ve_{g} = - K^\ve(x)\lambda_{g}(S^\ve_{g})
\bigg(\nabla p^\ve_{g} - \varrho_g(p^\ve_{g}) \vec{g}\bigg) \qquad
{\rm with}\,\, \lambda_{g}(s) = \frac{k_{r,g}}{\mu_{g}}
(s).
\end{equation}
Here $\vec g$, $\mu_w, \mu_g$ stand for the gravity vector and the viscosities of the
wetting and non-wetting phases, respectively.
We recall that the source term is neglected.

\medskip
In what follows we assume that the density of the wetting phase is a constant, without loss of generality
this constant is equal to $1$ so that ${\varrho_w(p^\ve_{w})} =1$.
We also assume that the density of non-wetting phase $\varrho_g$ is a continuously differentiable increasing function
such that
\begin{equation}
\begin{array}[c]{cl}
\varrho_g(p)=\varrho_{\rm min} \quad {\rm for} \,\, p \leqslant p_{\rm min};\qquad
\varrho_g(p)=\varrho_{\rm max} \quad {\rm for} \,\, p \geqslant p_{\rm max};
 \\[2mm]
\varrho_{\rm min}<\varrho_g(p)<\varrho_{\rm max} \,  \quad {\rm for} \,\,
p_{\rm min} <p < p_{\rm max} ;
\end{array}
\label{density-0}
\end{equation}
with
\begin{equation}
0 < \varrho_{\rm min} < \varrho_{\rm max} < +\infty \quad {\rm and} \quad
0 < p_{\rm min} < p_{\rm max} < +\infty.
\label{pro}
\end{equation}

By the definition of saturations,
\begin{equation}
S^\ve_{w} + S^\ve_{g} = 1 \qquad {\rm with } \,\, S^\ve_{w},\
S^\ve_{g} \geqslant 0.
\label{satura}
\end{equation}
Letting $S^\ve = S^\ve_{w}$ we get
\begin{equation}
S^\ve_{w}=S^\ve, \qquad S^\ve_{g}=1-S^\ve.
\label{satura1-1}
\end{equation}
Then the curvature of the contact surface between the two fluids links the jump
of pressure of two phases to the saturation by the capillary pressure law:

Then the capillary pressure law links the curvature of the contact surface between the two fluids,  the jump
of pressure of two phases and the saturation. It reads
\begin{equation}
\label{eq.pc}
P_{c}(S^\ve) = p^\ve_{g} - p^\ve_{w},
\end{equation}
where  $P_{c}(s)$ is a differentiable function such that $P^\prime_c(s) < 0$ for all $s \in [0, 1]$,
 and $P_{c}(1) = 0$.

From now on abusing slightly the notation we write $\lambda_g(S^\ve)$ instead of  $\lambda_g(1 - S^\ve)$.
Combining the above equations we arrive at the following system of equations:
\begin{equation}
\label{debut2-eps}
\left\{
\begin{array}[c]{ll}
\displaystyle
\Phi^\ve(x) \frac{\partial S^\ve}{\partial t} - {\rm div}\, \bigg\{K^\ve(x) \lambda_{w}(S^\ve)
\big(\nabla p^\ve_{w} - \vec{g}\big)\bigg\} = 0 \quad {\rm in}\,\, \mathcal{Q}_{T}; \\[5mm]
\displaystyle
\Phi^\ve(x) \frac{\partial \Theta^\ve}{\partial t} -
{\rm div}\, \bigg\{ K^\ve(x) \lambda_{g}(S^\ve)
\varrho_{g}(p^\ve_{g})\big(\nabla p^\ve_{g} - \varrho_{g}(p^\ve_{g})\vec{g} \big)\bigg\} = 0
\quad {\rm in}\,\, \mathcal{Q}_{T}; \\[5mm]
P_{c}(S^\ve) = p^\ve_{g} - p^\ve_{w} \quad {\rm in}\,\, \mathcal{Q}_{T},
\end{array}
\right.
\end{equation}
where
\begin{equation}
\label{theta-not-eps}
\Theta^\ve \eqdef \varrho_{g}(p^\ve_g)(1 - S^\ve).
\end{equation}
System \eqref{debut2-eps}--\eqref{theta-not-eps} is equipped  with the following boundary and initial conditions:\\[2mm]
{$\mathtt {Boundary\ conditions.}$} We suppose that
$\partial \mathcal{Q}$ consists of two $(d-1)$-dimensional sets $\Gamma_{\rm inj}$
and $\Gamma_{\rm imp}$ with a Lipschitz boundary such that  $\Gamma_{\rm inj} \cap \Gamma_{\rm imp} =
\emptyset$ and $\partial \Omega = \overline\Gamma_{\rm inj} \cup \overline\Gamma_{\rm imp}$.
The boundary conditions are given by:
\begin{equation}
\label{bc3-eps}
\left\{
\begin{array}[c]{ll}
p^\ve_{g}(x, t) = p^\ve_{w}(x, t) = 0 \quad {\rm on} \,\, \Gamma_{\rm inj} \times (0,T); \\[2mm]
\vec q^{\,\ve}_{w} \cdot \vec \nu = \vec q^{\,\ve}_{g} \cdot \vec \nu = 0 \quad {\rm on} \,\,
\Gamma_{\rm imp} \times (0,T),\\
\end{array}
\right.
\end{equation}
where the velocities $\vec q^{\,\ve}_{w}, \vec q^{\,\ve}_{g}$ are defined in \eqref{eq.qw}--\eqref{eq.qg}.\\[2mm]
$\mathtt{ Initial\ conditions}$. The initial conditions read:
\begin{equation}
\label{init1-eps}
p^\ve_{w}(x, 0) = p_{w}^{\bf 0}(x) \quad {\rm and} \quad
p^\ve_{g}(x, 0) = p_{g}^{\bf 0}(x) \quad {\rm in} \,\, \mathcal{Q}.
\end{equation}

Notice that from \eqref{bc3-eps} and \eqref{eq.pc} it follows that $S^\ve=1$ on $\Gamma_{\rm inj} \times (0,T)$.
The initial condition for $S^\ve$ is uniquely defined by the equation
\begin{equation}\label{snol-ini}
P_{c}(S^{\bf 0}(x)) = p_{g}^{\bf 0}(x) - p_{w}^{\bf 0}(x).
\end{equation}
Then according to \eqref{theta-not-eps} the initial condition for $\Theta^\ve$ reads
\begin{equation}
\label{theta-not-initi}
\Theta^{\bf 0} = \varrho_{g}(p^{\bf 0}_g)(1 - S^{\bf 0}).
\end{equation}

In the next section we specify the conditions on the data of system  \eqref{debut2-eps}--\eqref{theta-not-initi} which ensure
the existence of a solution to this system. Our goal is to study the asymptotic behaviour of this solution as $\ve\to0$,
and to construct the limit problem.

\subsection{Main assumptions}
\label{main-assump}

Here we provide our assumptions  on the data of system \eqref{debut2-eps}--\eqref{theta-not-initi}.
In order to formulate these assumptions we need two auxiliary functions. Namely, we denote
\begin{equation}
\label{gp+5}
\alpha(s) =
\frac{\lambda_{g}(s)\, \lambda_{w}(s)} {\lambda(s)} \left| P^\prime_{c}(s) \right|.
\end{equation}
and
\begin{equation}
\label{upsi-1}
\textstyle
\beta(s) = \int\limits_0^s \alpha(\xi)\, d\xi,
\end{equation}
We assume that the following conditions are fulfilled:
\begin{itemize}

\item[{\bf (A.1)}] The random variable $\boldsymbol{\Phi}$ belongs to $L^{\infty}(\Omega)$; moreover, and there are  constants
$\phi_-, \phi^+$ such that $0 < \phi_- < \phi^+$ and
\begin{equation}
\label{fifi}
0 < \phi_- \leqslant \boldsymbol{\Phi} \leqslant \phi^+ < 1 \quad {\rm a.\,s.\,\, in} \,\, \Omega.
\end{equation}

\item[{\bf (A.2)}] The random field $\boldsymbol{K}$ belongs to  $(L^{\infty}(\Omega))^{d\times d}$, and there exist
constants $K_-, K^+$ such that  $0 < K_- < K^+$ and
\begin{equation}
\label{tensork}
K_- |\xi|^2 \leqslant (\boldsymbol{K}(\omega)\xi, \xi) \leqslant K^+ |\xi|^2 \,\,
{\rm for \, all \, \xi \in \mathbb{R}^d, \,\, a.s. \, in}\,\, \Omega.
\end{equation}

\item[{\bf (A.3)}] The function $\varrho_{g} = \varrho_g(p)$ satisfies the conditions in (\ref{density-0}).

\item[{\bf (A.4)}] The capillary pressure $P_{c}(s)$ is a
$C^1([0, 1]; \mathbb{R}^+)$ function such that $P_{c}^\prime(s) < 0$ in $[0, 1]$
and $P_{c}(1) = 0$.

\item[{\bf (A.5)}] The functions $\lambda_{w}, \lambda_{g}$ are continuous
on the interval $[0, 1]$ and possess the following properties:
\begin{equation}\label{prop_lam_1}
0 \leqslant \lambda_{w}(s), \lambda_{g}(s) \leqslant 1\quad \hbox{for all } s\in[0, 1]; \qquad
\lambda_{w}(0) = \lambda_{g}(1) = 0;
\end{equation}
 there is a constant $L_0>0$ such that
 \begin{equation}\label{prop_lam_2}
 \lambda(s) =  \lambda_{w}(s) + \lambda_{g}(s)
\geqslant L_0 \quad \hbox{for all } s\in[0, 1].
\end{equation}
\item[{\bf (A.6)}] The function $\alpha$ defined in \eqref{gp+5}  belongs  to $C^1([0, 1]; \mathbb{R}^+)$.
Moreover,  $\alpha(s) > 0$ for $s\in(0, 1)$.  Notice that due to \eqref{prop_lam_1}  we have $\alpha (0) = \alpha (1) = 0$.

\item[{\bf (A.7)}]  The function $\beta^{-1}$, inverse of $\beta$ 
is H\"older continuous  on the interval
$[0, \beta(1)]$ that is there exist constants $\theta\in (0,1)$ and  $C_\beta>0$ such that for all
$s_1, s_2 \in [0, \beta(1)]$ the following inequality holds:
$$
\left|\beta^{-1}(s_1) - \beta^{-1}(s_2) \right| \leqslant C_\beta \, |s_1 - s_2|^\theta.
$$

\item[{\bf (A.8)}] The initial conditions  $p^{\bf 0}_{g}$ and
$p^{\bf 0}_{w}$ are  $L^2(\mathcal{Q})$ functions.

\item[{\bf (A.9)}] The function
$S^{\bf 0}$ satisfies the inequality $0 \leqslant S^{\bf 0} \leqslant 1$ a.e. in $\mathcal{Q}$.
\end{itemize}

In the existing literature conditions similar to those in {\bf (A.1)}--{\bf (A.9)} are commonly used in the theory of  multi-phase flow in porous media.

\subsection{Global pressure and fractional flow}
\label{gp-relat}

In this section, we  rearrange the system of equations in \eqref{debut2-eps}  using the notion of the so called {\it global pressure}. It was introduced in the works  \cite{Ant72, AKM83}.
Here we follow the approach developed in \cite{ant-kaz-mon1990,GC-JJ}, see also \cite{Chen-Huan-Ma-06}.
The main idea is to replace the studied two-phase flow with a flow of a fictive fluid for which  the
Darcy law holds with a non-degenerate coefficient. This rearrangements helps us to obtain several important a priori estimates
and finally the compactness results.\\
We are looking
for a pressure ${\mathsf P}^\ve$ and the coefficient $\gamma(s)$ such that $\gamma(s) > 0$
for all $s \in [0, 1]$, and
\begin{equation}
\label{gp-gam}
\lambda_{w}(S^\ve) \nabla p^\ve_{w} + \lambda_{g}(S^\ve) \nabla p^\ve_{g} = \gamma(S^\ve) \nabla {\mathsf P}^\ve.
\end{equation}
The global pressure ${\mathsf P}^\ve$ is defined by
\begin{equation}
\label{gp1}
p^\ve_{w} = {\mathsf P}^\ve + {\mathsf G}_{w}(S^\ve)
\quad {\rm and} \quad p^\ve_{g} = {\mathsf P}^\ve + {\mathsf G}_{g}(S^\ve)
\end{equation}
with
\begin{equation}
\label{gp3}
{\mathsf G}_{g}(s) = {\mathsf G}_{g}(0) +
\int\limits_0^{s} \frac{\lambda_{w}(r)} {\lambda(r)} \,P^\prime_{c}(r)\, dr,\qquad
{\mathsf G}_{w}(s) = {\mathsf G}_{g}(s) - P_{c}(s) ,
\end{equation}
and
\begin{equation}
\label{gp4}
\lambda(s) = \lambda_{w}(s) + \lambda_{g}(s).
\end{equation}
Due to \eqref{eq.pc} the function ${\mathsf P}^\ve$ is well defined.
Since
\begin{equation}
\label{gp4.1_biss}
\nabla {\mathsf G}_{w}(S^\ve) = - \frac{\lambda_{g}(S^\ve)} {\lambda(S^\ve)}
P^\prime_{c}(S^\ve)\, \nabla S^\ve,
\end{equation}
it is straightforward to check that
\begin{equation}
\label{gp2}
\lambda_{g}(S^\ve) \nabla {\mathsf G}_{g}(S^\ve) + \lambda_{w}(S^\ve) \nabla {\mathsf G}_{w}(S^\ve)
= 0
\end{equation}
and
$$
\lambda_{w}(S^\ve) \nabla p^\ve_{w} + \lambda_{g}(S^\ve) \nabla p^\ve_{g} = \lambda(S^\ve) \nabla {\mathsf P}^\ve +
\Big\{\lambda_{g}(S^\ve) \nabla {\mathsf G}_{g}(S^\ve) + \lambda_{w}(S^\ve) \nabla {\mathsf G}_{w}(S^\ve)\Big\}
=\lambda(S^\ve) \nabla {\mathsf P}^\ve.
$$
It remains to set $\gamma(s) = \lambda(s)$, and (\ref{gp-gam}) follows.

Notice that from (\ref{gp3}) we get:
\begin{equation}
\label{gp+grad}
\lambda_{w}(S^\ve) \nabla {\mathsf G}_{w}(S^\ve) = \alpha(S^\ve) \nabla S^\ve
\quad {\rm and} \quad
\lambda_{g}(S^\ve) \nabla {\mathsf G}_{g}(S^\ve) = - \alpha(S^\ve) \nabla S^\ve.
\end{equation}
It is also convenient to introduce the following quantities:
\begin{equation}
\label{bbb}
\mathfrak{a}(s) =
\sqrt{\frac{\lambda_{g}(s)\, \lambda_{w}(s)} {\lambda(s)}}\, \left| P^\prime_{c}(s) \right|\quad \hbox{and } \
\mathfrak{b}(s) = \int\limits_0^s \mathfrak{a}(\xi)\, d\xi.
\end{equation}
After straightforward computations, considering the definition of the global pressure,  (\ref{upsi-1}) and \eqref{bbb}, we obtain
\begin{equation}
\label{gp6-new}
\lambda_{g}(S^\ve) |\nabla p^\ve_{g}|^2 + \lambda_{w}(S^\ve) |\nabla p^\ve_{w}|^2
=
\lambda(S^\ve) |\nabla {\mathsf P}^\ve |^2 + \left|\nabla \mathfrak{b}(S^\ve) \right|^2
\end{equation}
and
\begin{equation}
\label{bbb-2}
\lambda_{w}(S^\ve) \nabla p^\ve_{w}
 = \lambda_{w}(S^\ve) \nabla {\mathsf P}^\ve + \nabla \beta(S^\ve),
\quad {\rm and} \quad \lambda_{g}(S^\ve) \nabla p^\ve_{g} = \lambda_{g}(S^\ve) \nabla {\mathsf P}^\ve -
\nabla \beta(S^\ve).
\end{equation}
Also, since by condition {\bf (A.5)} the functions $\lambda_{w}$ and  $\lambda_{g}$ are bounded, we have
\begin{equation}
\label{bbb-1}
\left|\nabla \beta(S^\ve) \right|^2=\frac{\lambda_{g}(S^\ve)\, \lambda_{w}(S^\ve)} {\lambda(S^\ve)}
\left|\nabla \mathfrak{b}(S^\ve) \right|^2
\leqslant C\,\left|\nabla \mathfrak{b}(S^\ve) \right|^2.
\end{equation}
It remains to determine the initial and boundary conditions for  ${\mathsf P}^\ve$.
 The initial condition can be easily derived from \eqref{init1-eps}, \eqref{snol-ini} and \eqref{gp1}. We leave the details to the reader.

Let us calculate the value of the global pressure
function ${\mathsf P}$ on $\Gamma_{\rm inj}$. Since by condition
{\bf (A.4)} we have $P_{c}(1) = 0$, then, by \eqref{eq.pc},  $S^\ve=1$
 on $\Gamma_{\rm inj}$. Therefore, thanks to \eqref{gp1},
 the function ${\mathsf P}^\ve$ is equal to a constant on $\Gamma_{\rm inj}$.
We denote it by ${\mathsf P}^1$.

\section{Existence result and estimates of a solution.}
\label{for-exist-result}

The question of the existence of a solution to problem  \eqref{debut2-eps}--\eqref{theta-not-initi} has been studied in the previous works \cite{CG-MS4},
\cite{our-siam}. In the same works a number of important a priori estimates have been obtained.
For the reader convenience we formulate here the corresponding existence result and several estimates for the solution.

Denote
$H^1_{\Gamma_{\rm inj}}(\Omega) = \left\{u \in H^1(\Omega) \,:\,
u = 0 \,\, {\rm on}\,\, \Gamma_{\rm inj} \right\}$.
We equip it with the norm $\Vert u \Vert_{H^1_{\Gamma_{\rm inj}}(\Omega)} =
\Vert \nabla u \Vert_{(L^2(\Omega))^d}$. 
\\
From \eqref{fifi} and \eqref{tensork} if follows that almost surely
\begin{equation}\label{realiz2}
 \phi_- \leqslant \Phi^\ve(x) \leqslant \phi^+  \quad\hbox{and}\quad K_- |\xi|^2 \leqslant (K^\ve(x)\xi, \xi) \leqslant K^+ |\xi|^2 \,\,
\hbox{ for all } x\in \mathcal{Q}, \ \ \xi\in \mathbb{R}^d,
\end{equation}
with
$\Phi^\ve(x)=\boldsymbol{\Phi}(\mathcal{T}_\frac x\ve\omega)$ and
$K^\ve(x)= \boldsymbol{K}(\mathcal{T}_\frac x\ve\omega)$, see \eqref{reali}.  We call $\omega\in\Omega$ for which
\eqref{realiz2} holds {\it typical}. From now on without mentioning it again we assume that $\omega$ is typical.

\begin{definition}\label{def_sol}{\rm
A triple of function $p^\ve_{g}=p^\ve_{g}(x,t)$,
 $p^\ve_{w}= p^\ve_{w}(x,t)$ and  $S^\ve=S^\ve(x,t)$ is called a solution of problem \eqref{debut2-eps}--\eqref{theta-not-initi} if
 \begin{equation}
\label{mainres-1}
p^\ve_{w}, p^\ve_{g} \in L^2(\mathcal{Q}_T) \quad \hbox{and}
\quad
\sqrt{\lambda_w(S^\ve)}\, \nabla p^\ve_{w},\
\sqrt{\lambda_g(S^\ve)}\, \nabla p^\ve_{g} \in L^2(\mathcal{Q}_T);
\end{equation}
\begin{equation}
\label{mainres-4}
\beta(S^\ve) \in L^2(0, T; H^1(\mathcal{Q})) \quad \hbox{and} \quad
{\mathsf P^\ve} - {\mathsf P}^1 \in L^2(0, T; H^1_{\Gamma_{\rm inj}}(\mathcal{Q}));
\end{equation}
\begin{equation}
\label{mainres-2}
\Phi^\ve \frac{\partial S^\ve}{\partial t}
\in L^2(0, T; H^{-1}(\mathcal{Q})) \quad \hbox{and} \quad
\Phi^\ve \frac{\partial \Theta^\ve}{\partial t}
\in L^2(0, T; H^{-1}(\mathcal{Q}));
\end{equation}
the maximum principle is valid for $S^\ve$:\\[-6mm]
\begin{equation}
\label{mainres-3}
0 \leqslant S^\ve \leqslant 1 \quad \hbox{a.e. in }
\mathcal{Q}_T,\qquad S^\ve=1 \quad \hbox{on } \Gamma_{\rm inj};
\end{equation}
 for any $\varphi_w, \varphi_g \in C^1([0, T]; H^1(\mathcal{Q}))$
such that $\varphi_w = \varphi_g = 0$ on $\Gamma_{\rm inj} \times (0,T)$ and
$\varphi_w(x,T) = \varphi_g(x,T) = 0$, we have:
\begin{equation}
\label{wf-1}
\begin{array}{c}
\displaystyle
-\int\limits_{\mathcal{Q}_T} \Phi^\ve(x) S^\ve \frac{\partial \varphi_w}{\partial t}
\,dx dt - \int\limits_{\mathcal{Q}} \Phi^\ve(x) S^{\bf 0}(x)
\varphi_w(x, 0)\, dx
+
\int\limits_{\mathcal{Q}_T} K^\ve(x) \lambda_w(S^\ve) \nabla p^\ve_{w} \cdot \nabla \varphi_w\, dx dt
\\[4mm]
\displaystyle
-
\int\limits_{\mathcal{Q}_T} K^\ve(x) \lambda_w(S^\ve)\, \vec{g} \cdot \nabla \varphi_w\, dx dt = 0
\end{array}
\end{equation}
and
\begin{equation}
\label{gf-2}
\begin{array}{c}
\displaystyle
-\int\limits_{\mathcal{Q}_T} \Phi^\ve(x) \Theta^\ve \frac{\partial \varphi_g}{\partial t} \,dx dt -
\int\limits_{\mathcal{Q}} \Phi^\ve(x) \Theta^{\bf 0}(x) \varphi_g(x, 0) \, dx \\[4mm]
\displaystyle
+ \int\limits_{\mathcal{Q}_T} K^\ve(x) \lambda_g(S^\ve)
\varrho_g(p^\ve_{g}) \nabla p^\ve_{g} \cdot \nabla \varphi_g\,dx dt
-
\int\limits_{\mathcal{Q}_T} K^\ve(x) \lambda_g(S^\ve)
\left[\varrho_g(p^\ve_{g}) \right]^2 \vec{g}\cdot \nabla \varphi_g\, dx dt
= 0
\end{array}
\end{equation}
with $\Theta^\ve$ defined in (\eqref{theta-not-eps});\\[2mm]
the following relation holds
$$
P_{c}(S^\ve) = p^\ve_{g} - p^\ve_{w};
$$
The initial conditions are satisfied in the following sense: for any  $\psi \in H^1_{\Gamma_{\rm inj}}(\mathcal{Q})$
\begin{equation}
\label{mainres-6}
\lim\limits_{t\to0}\ \ \int\limits_\mathcal{Q} \Phi^\ve(x) S^\ve(x,t)\, \psi(x) \, dx =
\int\limits_\mathcal{Q} \Phi^\ve(x) S^{\bf 0}(x) \psi(x) \, dx
\end{equation}
and
\begin{equation}
\label{mainres-7}
\lim\limits_{t\to0}\ \ \int\limits_\mathcal{Q} \Phi^\ve(x) \Theta^\ve(x,t) \psi(x) \, dx  =
\int\limits_\mathcal{Q} \Phi^\ve(x) \Theta^{\bf 0}(x) \psi(x) \, dx
\end{equation}
with $S^{\bf 0}$ and $\Theta^{\bf 0}$ defined in \eqref{snol-ini} and \eqref{theta-not-initi}, respectively.
}
\end{definition}

\begin{remark}{\rm
As was shown in \cite{our-siam}
for any function $\psi \in H^1_{\Gamma_{\rm inj}}(\mathcal{Q})$
the integrals
$\int\limits_{\mathcal{Q}} \Phi^\ve(x) S^\ve(x, t) \psi(x) dx$ and
$\int\limits_{\mathcal{Q}} \Phi^\ve(x) \Theta^\ve(x, t) \psi(x)dx$
are continuous functions of $t$ on $[0, T]$. Thus, the limits in
\eqref{mainres-6} and  \eqref{mainres-7} are well defined.}
\end{remark}

The following result has been proved in \cite{CG-MS4},
\cite{our-siam}.
\begin{theorem}
\label{th-main-exist}
Under  assumptions {\bf (A.1)}-{\bf (A.9)} for any $\ve>0$ problem  \eqref{debut2-eps}--\eqref{theta-not-initi}
has a solution  $p^\ve_{w}=p^\ve_{w}(x,t)$,
 $p^\ve_{g}= p^\ve_{g}(x,t)$ and  $S^\ve=S^\ve(x,t)$ that satisfies Definition \ref{def_sol}.
\end{theorem}

Below we also formulate several estimates for a solution of \eqref{debut2-eps}--\eqref{theta-not-initi} that have been obtained in \cite{our-siam}, \cite{app-nhm17}.

\begin{theorem}
\label{th-estima}
Let $p^\ve_{w}$, $p^\ve_{g}$, $S^\ve$ be a solution of problem  \eqref{debut2-eps},
and assume that the global pressure ${\mathsf P}^\ve$ is given by (\ref{gp1}).
Then
\begin{equation}
\label{hhh8-ved}
\int\limits_{\mathcal{Q}_T} \Big\{
\lambda_w(S^\ve) |\nabla p^\ve_{w}|^2 + \lambda_g(S^\ve) |\nabla p^\ve_{g}|^2 \Big\} \, dxdt
\leqslant C,
\end{equation}
\begin{equation}
\label{hhh9-ved}
\int\limits_{\mathcal{Q}_{T}} \Big\{ |\nabla {\mathsf P}^\ve|^2 +
|\nabla \beta(S^\ve)|^2 + \left|\nabla \mathfrak{b}(S^\ve) \right|^2  \Big\} \, dxdt \leqslant C,
\end{equation}
\begin{equation}
\label{nouv-100-ved}
\Vert \partial_t (\Phi^\ve \Theta^\ve) \Vert_{L^2(0,T;H^{-1}(\mathcal{Q}))}
+ \Vert \partial_t (\Phi^\ve S^\ve) \Vert_{L^2(0,T;H^{-1}(\mathcal{Q}))}
\leqslant C;
\end{equation}
here the constant $C$ is deterministic and does not depend on $\ve$.
\end{theorem}

\section{Homogenization result}
\label{sec_hm}
In this section we first remind the notion of stochastic two-scale convergence. We use here the realization-wise version
of this convergence that was introduced in \cite{PiZh2006}.
Then we provide several compactness results for a solution of problem \eqref{debut2-eps}--\eqref{theta-not-initi}.
After that we calculate the homogenized coefficients and formulate the homogenization theorem.
The proof of this theorem is given in the next section.

\subsection{Stochastic two-scale convergence. Compactness results.}
\label{ss_stoch2scale}
Changing if necessary the probability space we may assume that $\Omega$ is a compact metric space, $\mathcal{F}$ its Borel
$\sigma$-algebra and the dynamical system $T_x$ is continuous.
We give a definition of stochastic two-scale convergence that is adapted to our framework.

\begin{definition}\label{def_s2sc}
We say that a family of $L^2(\mathcal{Q}_T)$ functions $u^\ve=u^\ve_{\tilde\omega}(x,t)$ stochastically two-scale converges
to a function $u_{\tilde\omega}^0(x,t,\omega)$ if the following two conditions are fulfilled:
\begin{itemize}
  \item There exists $\ve_0>0$ such that
  $$
  \|u^\ve\|_{L^2(\mathcal{Q}_T)}\leq C_{\tilde\omega}  \quad\hbox{for all }\ve<\ve_0;
  $$
  \item Almost surely (for almost all  $\tilde\omega\in\Omega$) for any $\varphi\in C^\infty(\mathcal{Q}_T)$ and  any $\psi\in C(\Omega)$
  we have
\begin{equation}\label{ma_2s}
  \int\limits_{\mathcal{Q}_T} u_{\tilde\omega}
  ^\eps(x,t)\varphi(x,t)\psi(\mathcal{T}_\frac x\ve\tilde\omega)\,dxdt\longrightarrow
  \int\limits_{\mathcal{Q}_T}  u^0(x,t,\omega)\varphi(x,t)\psi(\omega)\,dxdtd\mathbf{P}(\omega).
\end{equation}
\end{itemize}
 \end{definition}

As was shown in \cite[Lemma 5.1]{PiZh2006}   for any function $\psi\in L^2(\Omega)$ there is its modification
(that is a function that differs from $\psi$ on the set of zero measure $\mathbf{P}$) such that relation  \eqref{ma_2s}
holds true. In what follows we consider this particular modification of  functions from $L^2(\Omega)$.

In order to formulate the main properties of stochastic two-scale convergence we introduce the subspaces $L^2_\mathrm{pot}(\Omega)$ and $L^2_\mathrm{sol}(\Omega)$ in the standard way, see \cite[Chapter x]{JKO94}.
Let $U_x$, $x\in\mathbb R^d$, be a strongly continuous group of unitary operators in $L^2(\Omega)$ defined by
$U_xf(\omega)=f(\mathcal{T}_x\omega)$. The generator of this group along the $j$th coordinate direction is denoted
by $\partial_j$ and its domain by $\mathcal{D}_j$. The set $\mathcal{D}=\bigcap_{j=1}^d\mathcal{D}_j$ is dense
in $L^2(\Omega)$.  Letting $\nabla_\omega u(\omega)=\big(\partial_1u(\omega),\ldots,\partial_d u(\omega)\big)$
for $u\in \mathcal{D}$ we denote by $L^2_\mathrm{pot}(\Omega)$ the closure of the set
$\{\nabla_\omega u\,:\, u\in\mathcal{D} \}$  in $(L^2(\Omega))^d$. The subspace $L^2_\mathrm{sol}(\Omega)$ is defined
as the closure in $(L^2(\Omega))^d$ of the set of vector function $(v_1(\omega),\ldots, v_d(\omega))$ such that
$v_j\in \mathcal{D}_j$, $j=1,\ldots, d$, and $\sum_1^d \partial_j v_j=0$.

The subspaces $L^2_\mathrm{pot}(\Omega)$ and  $L^2_\mathrm{sol}(\Omega)$ are orthogonal in $(L^2(\Omega))^d$,
and $(L^2(\Omega))^d=L^2_\mathrm{pot}(\Omega)\oplus L^2_\mathrm{sol}(\Omega)$. See, for instance, \cite{JKO94}
for further details.

Some properties of the stochastic two-scale convergence are collected in the following statement.
\begin{theorem}\label{t_stoch2s_prop}
For any family $u^\ve=u^\ve_{\tilde\omega}(x,t)$ such that $\|u^\ve\|_{L^2(\mathcal{Q}_T)}\leq C(\tilde\omega)$
there exists a sequence $\ve_k\to0$ and a function $u^0\in L^2(\mathcal{Q}_T\times \Omega)$ such that $u^{\ve_k}$
stochastically two-scale converges to $u^0$, as $k\to\infty$.\\[2mm]
If
$$
\|u^\ve\|_{L^2(\mathcal{Q}_T)}+\|\nabla_xu^\ve\|_{L^2(\mathcal{Q}_T)}\leq C(\tilde\omega),
$$
then $u^0$ does not depend on $\omega$,  $u^0\in L^2(0,T; H^1(\mathcal{Q}))$, and there exists a function
$u^1\in L^2(\mathcal{Q}_T;L^2_\mathrm{pot}(\Omega))$ such that
$$
\nabla_x u^\ve \mathop{\longrightarrow}\limits^{s2s}\nabla_x u^0 + u^1;
$$
here and later on symbol $\mathop{\longrightarrow}\limits^{s2s}$ denotes stochastic two-scale convergence.\\[2mm]
If
$$
\|u^\ve\|_{L^2(\mathcal{Q}_T)}+\ve\|\nabla_xu^\ve\|_{L^2(\mathcal{Q}_T)}\leq C(\tilde\omega),
$$
then
$$
\ve \nabla_xu^\ve  \mathop{\longrightarrow}\limits^{s2s} \nabla_\omega u^0(x,t,\omega).
$$
\end{theorem}
\noindent
The proof of these statements can be found in \cite{PiZh2006}.\\[2mm]
We turn to the properties of solutions of problem \eqref{debut2-eps}--\eqref{theta-not-initi}.
\begin{theorem}
\label{thm_comp}
Let $S^\ve$, $p_w^\ve$ and $p_g^\ve$ be a solution of problem \eqref{debut2-eps}--\eqref{theta-not-initi}, and assume that conditions {\bf (A.1)}--{\bf (A.9)} are fulfilled. Then there exist a function $\widehat S=\widehat S(x,t)$, $0\leqslant S\leqslant 1$, a function
$\widehat{\mathsf P}\in L^2(0,T;H^1(\mathcal{Q}))$ and a function $\widehat\Theta\in L^{\infty}(\mathcal{Q}_T)$ such that, for a subsequence, as $\eps\to0$,
\begin{equation}
\label{conve_1}
S^\ve(x, t) \to \widehat S(x, t) \quad \hbox{\rm in } L^q(\mathcal{Q}_T)\quad
\hbox{\rm for all } q\in[1, +\infty);
\end{equation}
\begin{equation}
\label{conve_2}
{\mathsf P}^\ve(x, t) \rightharpoonup \widehat{\mathsf P}(x, t) \,\, {\rm weakly\,\, in}\,\, L^2(0, T; H^1(\mathcal{Q}));
\end{equation}
\begin{equation}
\label{conve_3}
\Theta^\ve \to \widehat\Theta  \quad \hbox{\rm in }
L^2(\mathcal{Q}_T).
\end{equation}
Moreover, 
$\widehat\Theta = (1 -\widehat S)\, \varrho_{g}(P_{g})$ with $P_{g}=\widehat{\mathsf P}+G_g(\widehat S)$.
\end{theorem}
\begin{remark}
The statement of the  latter theorem holds for any typical realization $\tilde\omega$.
However, the choice of a convergent subsequence as well as the limit functions $\widehat S$, $ \widehat{\mathsf P}$
and $\widehat\Theta$ might depend on $\tilde\omega$.
\end{remark}
As an immediate consequence of \eqref{conve_1} we have
\begin{equation}
\label{conve_4}
\beta(S^\ve) \to \beta(\widehat S) \quad \hbox{ in } L^q(\Omega_T)\quad
\hbox{for all } q \in[1, +\infty).
\end{equation}
\begin{proof}[Proof of Theorem \ref{thm_comp}]
By the Birkhoff ergodic theorem almost surely the functions $\Phi^\ve$ converges weakly in $L^2(\mathcal{Q}_T)$
to a constant equal to $\mathbf{E}\boldsymbol{\Phi}=\int_\Omega\boldsymbol{\Phi}(\omega)\,d\mathbf{P}(\omega)$.
Then  according to Lemma 4.2 and Remark 1 in \cite{our-siam} the families $\{S^\ve \}_{\ve>0}$ and $\{\Theta^\ve\}_{\ve>0}$
are compact in   $L^2(\mathcal{Q}_T)$. This implies the desired convergence in \eqref{conve_1} and \eqref{conve_3}.
The convergence in \eqref{conve_2} is an immediate consequence of estimate \eqref{hhh9-ved}.\\
The relation $\widehat\Theta = (1 -\widehat S)\, \varrho_{g}(P_{g})$
has been  justified  in Lemma 4.8 in \cite{our-siam}.
\end{proof}


\subsection{Effective system and homogenization theorem}
\label{hom-res}

We begin this section by considering an auxiliary problem that reads:\\
given a vector $\eta\in\mathbb R^d$ find $\boldsymbol{\xi}_\eta\in L^2_{\mathrm{pot}}(\Omega)$ such that
$$
\boldsymbol{K}(\boldsymbol{\xi}_\eta+\eta)\in L^2_{\mathrm{sol}}(\Omega).
$$
This problem has a unique solution, see \cite[Chapter x.x]{JKO94}.
If $\eta$ is equal to the $j$-th coordinate vector $e_j$ in $\mathbb R^d$, we denote the corresponding solution by
$ \boldsymbol{\xi}_j$.
\\[2mm]
Let $ \boldsymbol{\xi}=\boldsymbol{\xi}(\omega)$ be a matrix valued function whose $j$-th column coincides
with $ \boldsymbol{\xi}_j$, $j=1,\ldots, d$.
We define the effective characteristics
$$
K^\mathrm{hom}=\int_\Omega \boldsymbol{K}(\omega)\big(\boldsymbol{\xi}(\omega)+\mathbf{I}\big)\,d\mathbf{P}(\omega),
\quad
\Phi^\mathrm{hom}=\int_\Omega \boldsymbol{\Phi}(\omega)\,d\mathbf{P}(\omega);
$$
here the symbol $\mathbf{I}$ stands for the unit matrix.   \\
The homogenized system takes the form
\begin{equation}
\label{hom-0}
\left\{
\begin{array}[c]{ll}
\displaystyle
\Phi^\mathrm{hom} \, \dfrac{\partial \widehat{S}}{\partial t} -
{\rm div}_x \Big\{K^\mathrm{hom}\, \lambda_{w}(\widehat{S}) \big[ \nabla P_{w} - \vec{g}\big] \Big\}
= 0 \quad \hbox{in }  \mathcal{Q}_T; \\[4mm]
\displaystyle
\Phi^\mathrm{hom} \, \dfrac{\partial \widehat{\Theta}}{\partial t} -
{\rm div}_x \bigg\{K^\mathrm{hom}\, \varrho_{g}(P_{g})\,\lambda_{g}(\widehat{S}) \big[
\nabla P_{g} - \varrho_{g}(P_{g}) \vec{g}\big] \bigg\} = 0
\quad {\rm in} \,\, \mathcal{Q}_T;\\[4mm]
P_c(\widehat{S}) = P_g - P_{w} \quad \hbox{in }  \mathcal{Q}_T,\\[4mm]
\widehat{\Theta}=(1 - \widehat{S})\, \varrho_{g}(P_{g}) \quad \hbox{in }  \mathcal{Q}_T,\\[4mm]
0 \leqslant \widehat{S} \leqslant 1 \quad {\rm in} \,\, \mathcal{Q}_T.
\end{array}
\right.
\end{equation}
{$\mathtt {Boundary\ conditions.}$}
\begin{equation}
\label{hom-5}
\left\{
\begin{array}[c]{ll}
P_{g}(x, t) = P_{w}(x, t) = 0 \quad \hbox{on } \Gamma_{\rm inj} \times (0,T), \\[2mm]
\widehat{\vec q}_{w} \cdot \vec \nu = \widehat{\vec q}_{g} \cdot \vec \nu = 0 \quad \hbox{on }
\Gamma_{\rm imp} \times (0,T);\\
\end{array}
\right.
\end{equation}
where the velocities $\widehat{\vec q}_{w},\,\widehat{\vec q}_{g}$ are defined by
\begin{equation}
\widehat{\vec q}_{w} = - {K}^{\mathrm{hom}} \lambda_{w}(\widehat{S})
\bigg(\nabla P_{w} - \vec{g}\bigg) \quad
\hbox{and} \quad \widehat{\vec q}_{g}= - {K}^{\mathrm{hom}} \lambda_{g}(\widehat{S})
\Big(\nabla P_{g} - \varrho_g(P_{g}) \vec{g}\Big).
\label{hom-6}
\end{equation}

{$\mathtt {Initial \ conditions.}$} The initial conditions are the same as for the original system in \eqref{init1-eps}. Namely,
\begin{equation}
\label{init1-hom}
P_{w}(x, 0) = p_{w}^{\bf 0}(x) \quad \hbox{and} \quad
P_{g}(x, 0) = p_{g}^{\bf 0}(x) \quad \hbox{in }  \mathcal{Q}.
\end{equation}

Observe that the limit problem is deterministic.
The functions $\widehat{S}$, $P_{w}$ and $P_g$ represent the homogenized wetting phase saturation, wetting phase pressure and non-wetting phase pressure, respectively.

\begin{theorem}
\label{main-hom-theor}
Assume that conditions {\bf (A.1})--{\bf (A.9)} hold. Then almost surely, a solution $\big(S^\ve, \,p^\ve_w, \, p^\ve_g\big)$
of problem \eqref{debut2-eps}--\eqref{theta-not-initi} converges for a subsequence, as $\ve\to0$, to a solution
$\big(\widehat{S}, \,P_w, \, P_g\big)$ of  the homogenized problem in  \eqref{hom-0}--\eqref{init1-hom} in the following topology:
$$
S^\ve \,\to\, \widehat{S} \quad\hbox{\rm in } L^q(\mathcal{Q}_T) \quad\hbox{\rm for any }q\in[1,+\infty);
$$
$$
p^\ve_w\,\rightharpoonup\, P_w,\quad \hbox{\rm and}\quad p^\ve_g\,\rightharpoonup\, P_g \ \ \hbox{\rm weakly in }
L^2(\mathcal{Q}_T).
$$
\end{theorem}

The proof of this theorem is given in the next section.


\subsection{Proof of homogenization theorem}
\label{proof-main-hom-th}

\begin{proof}[Proof of Theorem  \ref{main-hom-theor}]
The rigorous derivation of the limit problem relies on the above a priori
estimates and compactness results as well as on stochastic two-scale convergence
technique developed in \cite{PiZh2006}.

By the estimates in Theorem \ref{th-estima} and Theorem \ref{t_stoch2s_prop} we obtain
that almost surely for a subsequence
\begin{equation}\label{sto2sca1}
\lambda_w(S^\ve)\nabla p^\ve_w= \lambda_{w}(S^\ve) \nabla {\mathsf P}^\ve + \nabla \beta(S^\ve)
\mathop{\rightharpoonup}\limits^{s2s}
 \lambda_{w}(\widehat{S}) \nabla \widehat{\mathsf P} + \nabla \beta(\widehat{S})
 +\theta_w
\end{equation}
with $\theta_w=\theta_w(x,t,\omega)$,  $\theta_w\in L^2(\mathcal{Q}_T; L^2_{\mathrm{pot}}(\Omega))$,
and
\begin{equation}\label{sto2sca2}
\lambda_g(S^\ve)\nabla p^\ve_g= \lambda_{w}(S^\ve) \nabla {\mathsf P}^\ve + \nabla \beta(S^\ve)
\mathop{\rightharpoonup}\limits^{s2s}
 \lambda_{g}(\widehat{S}) \nabla \widehat{\mathsf P} + \nabla \beta(\widehat{S})
 +\theta_g
\end{equation}
with $\theta_g=\theta_g(x,t,\omega)$,  $\theta_g\in L^2(\mathcal{Q}_T; L^2_{\mathrm{pot}}(\Omega))$
\begin{lemma}
Under our standing assumptions, for a subsequence
\begin{equation}\label{sto2sca3}
K^\ve \big(\lambda_{w}(S^\ve) \nabla {\mathsf P}^\ve + \nabla \beta(S^\ve)\big)
\mathop{\rightharpoonup}\limits^{s2s}
 \boldsymbol{K}\big(\lambda_{w}(\widehat{S}) \nabla \widehat{\mathsf P} + \nabla \beta(\widehat{S})
 +\theta_w\big),
\end{equation}
\begin{equation}\label{sto2sca4}
K^\ve \varrho_g({\mathsf P}^\ve + {\mathsf G}_{g}(S^\ve))\big(\lambda_{g}(S^\ve) \nabla {\mathsf P}^\ve + \nabla \beta(S^\ve)\big)
\mathop{\rightharpoonup}\limits^{s2s}
 \boldsymbol{K}\varrho_g(\widehat{\mathsf P} + {\mathsf G}_{g}(\widehat{S}))\big(\lambda_{g}(\widehat{S}) \nabla \widehat{\mathsf P} + \nabla \beta(\widehat{S})
 +\theta_g\big).
\end{equation}
\end{lemma}
\begin{proof}
Since the function $K^\ve$ is statistically homogeneous and bounded,  the limit relation in \eqref{sto2sca3} immediately follows from \eqref{sto2sca1}.   Justification of the convergence in \eqref{sto2sca4} is more tricky. Denote
 $\{\widehat S=1\}$ the set $\{(x,t)\in\Omega_T\,:\,\widehat S(x,t)=1\}$,  and let
$\mathbf{1}_{\{\widehat S=1\}}$ be the corresponding characteristic function. From \eqref{conve_3}it is easy to deduce that
 $\big(1-\mathbf{1}_{\{\widehat S=1\}}\big) \varrho_g({\mathsf P}^\ve + {\mathsf G}_{g}(S^\ve))$ converges to
$\big(1-\mathbf{1}_{\{\widehat S=1\}}\big)\varrho_g(\widehat{\mathsf P} + {\mathsf G}_{g}(\widehat{S}))$
a.e., as $\ve\to0$. Considering  the boundedness of $\varrho_g$ and the properties of $K^\ve$ we conclude that
\begin{equation}\label{sto2sca6}
\begin{array}{rl}
\displaystyle
\big(1-\mathbf{1}_{\{\widehat S=1\}}\big)K^\ve \varrho_g({\mathsf P}^\ve +\!\!\!\!\!& {\mathsf G}_{g}(S^\ve))
\big(\lambda_{g}(S^\ve) \nabla {\mathsf P}^\ve + \nabla \beta(S^\ve)\big) \mathop{\rightharpoonup}\limits^{s2s}
\\[2mm]
\displaystyle
& \mathop{\rightharpoonup}\limits^{s2s}
\big(1-\mathbf{1}_{\{\widehat S=1\}}\big) \boldsymbol{K}
\varrho_g(\widehat{\mathsf P} + {\mathsf G}_{g}(\widehat{S}))
\big(\lambda_{g}(\widehat{S}) \nabla \widehat{\mathsf P} + \nabla \beta(\widehat{S}) +\theta_g\big).
\end{array}
\end{equation}
It remains to show that
\begin{equation}\label{sto2sca7}
\begin{array}{rl}
\displaystyle
\mathbf{1}_{\{\widehat S=1\}}K^\ve \varrho_g({\mathsf P}^\ve +\!\!\!\!\!& {\mathsf G}_{g}(S^\ve))
\big(\lambda_{g}(S^\ve) \nabla {\mathsf P}^\ve + \nabla \beta(S^\ve)\big) \mathop{\rightharpoonup}\limits^{s2s}
\\[2mm]
\displaystyle
& \mathop{\rightharpoonup}\limits^{s2s}
\mathbf{1}_{\{\widehat S=1\}} \boldsymbol{K}
\varrho_g(\widehat{\mathsf P} + {\mathsf G}_{g}(\widehat{S}))
\big(\lambda_{g}(\widehat{S}) \nabla \widehat{\mathsf P} + \nabla \beta(\widehat{S}) +\theta_g\big).
\end{array}
\end{equation}
Since $\lambda_g(1)=0$ and $\nabla \widehat{\mathsf P}^\ve$ is bounded in $L^2(\mathcal{Q}_T)$,
$$
\mathbf{1}_{\{\widehat S=1\}}\lambda_{g}(S^\ve) \nabla \widehat{\mathsf P}^\ve \longrightarrow 0
 =\mathbf{1}_{\{\widehat S=1\}}\lambda_{g}(\widehat{S}) \nabla \widehat{\mathsf P} \quad\hbox{strongly in }L^2(\mathcal{Q}_T).
$$
By \eqref{hhh9-ved} and the first relation in \eqref{bbb-1} we obtain
$$
\mathbf{1}_{\{\widehat S=1\}}\nabla \beta(S^\ve) \longrightarrow 0
 =\mathbf{1}_{\{\widehat S=1\}}\nabla \beta(\widehat{S})  \quad\hbox{strongly in }L^2(\mathcal{Q}_T).
$$
Therefore, $\mathbf{1}_{\{\widehat S=1\}}\theta_g=0$. Combining the last three relations yields \eqref{sto2sca7} and completes the proof of Lemma.
\end{proof}
Next we choose in the integral identity \eqref{gf-2} a test function of the form $\varphi_g(x,t)=\ve \phi(x,t)\psi(T_\frac x\ve \omega)$ with $\phi\in C^\infty(\mathbb R^d\times[0,T])$ that has a compact support in $\mathbb R^d\times[0,T)$,
and $\psi\in \mathcal{D}(\Omega)$. Then the first two integrals on the left-hand side of  \eqref{gf-2} tend to zero as $\ve\to0$.
Passing to the two-scale limit in the last two integrals  and considering \eqref{sto2sca4} we obtain
\begin{equation}
\label{limlim-ggg}
\begin{array}{c}
\displaystyle
 \int\limits_{\mathcal{Q}_T}\int\limits_\Omega \phi\boldsymbol{K}(\omega)
\varrho_g(\widehat{\mathsf P}+ {\mathsf G}_{g}(\widehat{S}))\big(\lambda_{g}(\widehat{S}) \nabla \widehat{\mathsf P} + \nabla \beta(\widehat{S})
 +\theta_g\big) \cdot \nabla_\omega \psi(\omega)\,dx dtd\mathbf{P}(\omega)\\[4mm]
 \displaystyle
-\int\limits_{\mathcal{Q}_T}\int\limits_\Omega  \phi\boldsymbol{K}(\omega) \lambda_g(\widehat S)
\left[\varrho_g(\widehat{\mathsf P}+ {\mathsf G}_{g}(\widehat{S})) \right]^2 \vec{g}\cdot
\nabla_\omega \psi(\omega)\, dx dt\mathbf{P}(\omega)
= 0
\end{array}
\end{equation}
Since $\phi$ is an arbitrary smooth function with a compact support, then for almost all $(x,t)\in\mathbb R^d\times[0,T]$
we have
$$
 \int\limits_\Omega \boldsymbol{K}(\omega)
\big[\lambda_{g}(\widehat{S}) \nabla \widehat{\mathsf P} + \nabla \beta(\widehat{S})
-\lambda_g(\widehat S)
\varrho_g(\widehat{\mathsf P}+ {\mathsf G}_{g}(\widehat{S}))  \vec{g}+\theta_g\big]\cdot
\Psi(\omega)\, \mathbf{P}(\omega)
= 0
$$
for each $\Psi\in L^2_{\rm pot}(\Omega)$. Taking into account the definition of $\boldsymbol{\xi}(\cdot)$ we arrive
at the following formula
\begin{equation}
\label{limlim-gugu}
\theta_g=\boldsymbol{\xi}(\omega)\big[ \lambda_{g}(\widehat{S}) \nabla \widehat{\mathsf P}+
\nabla \beta(\widehat{S})
-\lambda_g(\widehat S)
\varrho_g(\widehat{\mathsf P}+ {\mathsf G}_{g}(\widehat{S}))\big].
\end{equation}
It remains to choose a smooth test function $\varphi$ of the form $\varphi= \varphi(x,t)$ with a compact support
in $\mathbb R^d\times[0,T)$.   Considering \eqref{limlim-gugu} and passing to the two-scale limit in \eqref{gf-2} as
$\ve\to0$ yields the weak formulation of the second equation in \eqref{hom-0}. The first equation can be derived in a similar way with a number of simplifications.
The proof of the fact that the boundary and the initial conditions in \eqref{hom-5}--\eqref{init1-hom} are fulfilled is straightforward.
This completes the proof of Theorem  \ref{main-hom-theor}.
\end{proof}


\section*{Acknowledgments}
An essential part of this work was done during the visit of  A. Piatnitski at the Applied Mathematics Laboratory of the University of Pau in May 2019. The financial support of the visit and the hospitality of the people are gratefully
acknowledged.

\renewcommand{\baselinestretch}{0.8}

\end{document}